\documentclass[a4paper,oneside,10pt]{article}%
\usepackage{threeparttable}
\usepackage{setspace}
\usepackage{lipsum}
\usepackage{amsfonts}
\usepackage{algorithm}
\usepackage{algorithmic}
\usepackage{diagbox}
\usepackage{setspace}
\usepackage{graphicx}
\usepackage{epstopdf}
\usepackage{graphicx,amsfonts,amssymb,mathrsfs,amsmath,color,fancyhdr}
\usepackage{multirow}
\usepackage{psfrag}
\usepackage{graphics}
\usepackage{float}
\usepackage{multirow}
\usepackage{enumerate}
\usepackage{caption}
\usepackage{booktabs}
\usepackage{array}
\usepackage{subfig}
\usepackage{subeqnarray}
\usepackage{cases}
\usepackage[pagewise]{lineno}
\allowdisplaybreaks[4]
\ifpdf
  \DeclareGraphicsExtensions{.eps,.pdf,.png,.jpg}
\else
  \DeclareGraphicsExtensions{.eps}
\fi
\providecommand{\U}[1]{\protect \rule{.1in}{.1in}}

\newcommand{\be}{\begin{equation}}
\newcommand{\ee}{\end{equation}}

\newtheorem{lem}{Lemma}[section]

\newtheorem{thm}{Theorem}[section]

\newtheorem{rem}{Remark}[section]

\newtheorem{ex}{Example}[section]

\usepackage{amsopn}

\newenvironment{proof}[1][Proof]{\noindent \textbf{#1.} }{\  \rule{0.5em}{0.5em}}
\numberwithin{equation}{section}
\begin{document}

\title{Parareal algorithm via Chebyshev-Gauss spectral collocation method}
\author{Quan Zhou\thanks{College of Science, National University of Defense Technology,
Changsha, Hunan 410073, China. quanzhoujune@163.com.
Zhou is supported by the the Postgraduate Scientific Research Innovation Project of Hunan Province (No.CX20210012).}
\and Yicheng Liu\thanks{College of Science, National University of Defense Technology,
Changsha, Hunan 410073, China. liuyc2001@hotmail.com.}
\and Shu-Lin Wu\thanks{Corresponding author, School of Mathematics and Statistics,
Northeast Normal University, Changchun 130024, China. wushulin84@hotmail.com.}}
\date{}
\maketitle

\textbf{Abstract}.
We present the Parareal-CG algorithm for time-dependent differential equations in this work. 
The algorithm is a parallel in time iteration algorithm utilizes
Chebyshev-Gauss spectral collocation method for fine propagator $\mathcal{F}$
and backward Euler method for coarse propagator $\mathcal{G}$.
As far as we know, this is the first time that the
spectral method used as the $\mathcal{F}$ propagator of the parareal algorithm.
By constructing the stable function of the Chebyshev-Gauss spectral collocation method
for the symmetric positive definite (SPD) problem,
we find out that the Parareal-CG algorithm and the Parareal-TR algorithm,
whose $\mathcal{F}$ propagator is chosen to be a trapezoidal ruler, converge similarly, 
i.e., the Parareal-CG algorithm converge as fast as Parareal-Euler algorithm with sufficient Chebyhsev-Gauss points in every coarse grid.
Numerical examples including ordinary differential equations and time-dependent partial differential equations
are given to illustrate the high efficiency and accuracy of the proposed algorithm.
\newline


\textbf{MSC-classification}. 65L05, 65L20, 65L60, 68Q60.
\section{Introduction}
%
We are considering utilizing the parareal algorithm for the initial-value problems in the following
\begin{equation}\label{IVP}
\left\{\begin{aligned}
&\frac{du}{dt}=\;f(t,u),\quad t\in[0,T],\\
&u(0)=\;u_0.
\end{aligned}\right.\end{equation}
where $f :(0,T)\times \mathbb{R}^m\rightarrow\mathbb{R}^m$, and $u_0\in\mathbb{R}^m$.
Parareal is a well-studied parallel in time algorithm
developed by Lions, Maday, and Turinici in 2001 \cite{LMT2001}.
The algorithm obtains the solution in a limited number of predictor-corrector iterations
utilizing random initial values at each temporal subinterval, stopping when a tolerance is reached.
The global error produced by this iterative method is equivalent to
that obtained by the serial of the fine propagator.
Due to the advantages of parareal algorithm including but not limited to efficiency and convergency,
many relevant methods have emerged in recent years \cite{G2015, LMT2001, PTSA2022, WZ2021}.
Together with these methods, the parareal algorithm has been identified in various fields of research, including
optimal control problems \cite{FVM2022, MSS2010, WH2018, WL2020},
wave equations \cite{DW2021, GW2020}, stochastic differential equations \cite{BW2020, HWZ2019, ZWZLZ2020},
Hamiltonian systems \cite{DBLM2013, GH2014}, incompressible flows \cite{CSdSF2022, DSW2022},
heat equations \cite{JL2023, LDT2022, WMOZ2022}, algebraic equations \cite{GKS2022, LA1993},
molecular dynamics \cite{BBMTZ2022, LLS2022},
and other partial differential equations \cite{BJMO2022, LW2022, LWWZ2022}.

The parareal algorithm combines two numerical methods which are the
coarse propagator $\mathcal{G}$ and fine propagator $\mathcal{F}$,
associated with the large time step $\Delta T$ and small time step $\Delta t$, respectively.
The ratio $J=\Delta T/\Delta t$ is assumed to be greater than 1. 
The $\mathcal{G}$ propagator is often chosen to be backward Euler method which is inexpensive and strongly stable 
so that it is available for the large time step $\Delta T$ computations. 
Additionally, numerical methods based on Taylor's expansion and quadrature formula
which is much cheaper are frequently chosen as $\mathcal{F}$ propagators,
and the convergence of various $\mathcal{F}$ propagators
has been thoroughly investigated for the symmetric positive definite (SPD) problem.
\begin{equation}\label{SPD}
u'(t)+Au(t)=g(t),\quad A\in\mathbb{R}^{m\times m}\; \text{symmetric positive definite (SPD)}.
\end{equation}
Gander and Vandewalle \cite{GV2007} proposed the convergence theorem of parareal algorithm and
illustrated numerically that Parareal-Euler algorithm using for $\mathcal{F}$ the backward-Euler method converges rapidly.
Based on which, Mathew, Sarkis and Schaerer \cite{MSS2010} proved theoretically
that the parareal-Euler algorithm converged robustly and the convergence factor is $0.298$ for $J\geq 2$.
Wu \cite{W2015} demonstrated that the robust convergence also holds for $\mathcal{F}$ propagator
chosen to be second-order diagonal implicit Runge-Kutta (DIRK2) method
and TR/BDF2 method (ode23tb solver for ODEs in Matlab).
For $J\geq 2$, the convergence factor of the Parareal-2sDIRK and Parareal-TR/BDF2 are $0.316$ and $0.333$ respectively.
Wu and Zhou \cite{WZ2015} also showed the analysis $\mathcal{F}$ propagator for
the third-order diagonal implicit Runge-Kutta method with a convergence factor $0.333$ for $J\geq4$;
they also proved that for trapezoidal formula and fourth-order Gauss-Runge-Kutta integrator chosen to be $\mathcal{F}$ propagator,
there exists a $J_{\min}^*$ depends on both spectral radius of $A$ and the step size $\Delta t$
which make the convergence factor be $0.333$ for $J\geq J_{\min}^*$.
Recently, Yang, Yuan and Zhou \cite{YYZ2022} gave a more general result
stating that if the $\mathcal{F}$ propagator is strongly stable single step integrators,
there must exist a positive $J_{\min}^*$
(independent of step sizes $\Delta T$, $\Delta t$, terminal time $T$, problem data $u_0$ and $f$,
as well as the spectral radius of $A$) such that
the parareal algorithm converges linearly with convergence factor close to $0.3$ for all $J\geq J_{\min}^*$.

In this work, 
we choose Chebyshev-Gauss spectral collocation method \cite{YW2015}
for fine propagator and present Parareal-CG algorithm.
The Chebyshev-Gauss spectral collocation method is an overall iteration method 
with $M$ Chebyshev-Gauss points in the computation interval,  
so it combines the advantages of spectral accuracy and computational efficiency.
Also, it is unnecessary to solve implicit equations for every fine time step $\Delta t$
as classical methods via Taylor's expansion mentioned in \cite{GV2007, W2015, WZ2015, YYZ2022}.
We would briefly introduce the spectral collocation methods for ODEs.
Clenshaw and Norton first presented the Chebyshev-Picard method
for solving nonlinear ordinary differential equations in 1963,
in which the nonlinear term was approximated by Chebyshev series,
so that the method was collocated at Chebyshev-Gauss-Lobatoo points and implemented by Picard iteration.
Later, a matrix-vector form of the method was introduced by Feagin and Nacozy \cite{FN1983},
greatly increases the computation efficiency.
Yang and Wang \cite{YW2015} proposed Chebyshev-Gauss spectral collocation method
via Chebyshev-Gauss points for ODEs in a single interval and analyzed the convergence by the $hp$ version.
The method demonstrated significant advantages in astrodynamics simulations \cite{B2010, BJ2011, WJ2019, WNL2020}
because of its spectral accuracy and computational efficiency.
Additionally, some relevant spectral collocation methods for solving ODEs 
are also proposed in recent years \cite{GW2009, GW2010, WM2016}.

We present the matrix-vector form of Chebyshev-Gauss spectral collocation method
by the approach in \cite{FN1983}.
Based on which we construct its stability function
and illustrate the convergence of the Parareal-CG method numerically.
We find out there exists a $M_{\min}^*$ depends on
spectral radius of $A$ and the step size $\Delta T$
which makes the convergence factor be $0.333$ for $M \geq M_{\min}^*$.
The proposed method has the same convergence as Parareal-TR and Parareal-Gauss4 methods.

The rest of the paper is organized as follows,
in section \ref{sec:reCG} we recall the Chebyshev-Gauss spectral collocation method and its convergence theorem.
Parareal-CG algorithm is proposed in \ref{sec:para-CG}.
We provide stability function of Chebyshev-Gauss spectral collocation method
and observe the convergence of Parareal-CG algorithm in \ref{sec:con}.
Several numerical experiments are carried out in Section \ref{sec:num}
to demonstrate the high accuracy and convergency of the proposed method.
We finally give some conclusions in Section \ref{sec:conclusion}.

\section{Revisit of Chebyshev-Gauss Spectral Collocation Method}\label{sec:reCG}

In this section, we revisit the Chebyshev-Gauss spectral collocation method
proposed by Yang and Wang \cite{YW2015}
in a general interval $[a,b]$, ($b>a>0$) with the initial condition $u(a)=u_a$.

Let $T_l(\tau)=\cos(l \arccos(\tau))$ be the standard Chebyshev polynomials of degree $l$, ($l=0,1,\cdots$) with $\tau\in[-1,1]$,
then by using the affine transformation 
we can define the shifted Chebyshev polynomials
\begin{equation}\label{shifted_Che}
\widetilde{T}_{l}(t)=T_l\left(\frac{2(t-a)}{b-a}-1\right),\quad t\in[a,b],\quad l = 0,1,\cdots
\end{equation}
According to the definition, one can obtain the following shifted Chebyshev derivative relationship directly
\begin{equation}\label{Che_diff}\begin{split}
&\widetilde{T}_{0}'(t)=\;0, \quad\quad \widetilde{T}_1'(t)=\;\frac{2}{b-a},\\
&\frac{1}{l+1}\widetilde{T}_{l+1}'(t)-\frac{1}{l-1}\widetilde{T}_{l-1}'(t)
=\;\frac{4}{b-a}\widetilde{T}_l(t),\quad l\geq 2.
\end{split}\end{equation}
Let $\tau_m$ denote the standard Chebyshev-Gauss (CG) points in $(-1,1)$, 
\begin{equation}
\tau_m=-\cos\frac{(2m+1)\pi}{2M+2}, \quad m=0,1,\cdots,M,
\end{equation}
the corresponding shifted Chebyshev-Gauss (CG) points $t_m$ has the form
\begin{equation}\label{shifted_CG}
t_m = \frac{b-a}{2}\tau_m+\frac{a+b}{2}, \quad m=0,1,\cdots,M,
\end{equation}
which are the zeros of $\widetilde{T}_{M+1}(t)$.

Denote $\mathcal{P}_{M+1}(a,b)$ be a set of polynomials of degree at most $M+1$ in $(a,b)$,
the Chebyshev-Gauss spectral collocation method is to seek
$u_M(t)\in\mathcal{P}_{M+1}(a,b)$ defined by
\begin{equation}\label{uM}
u_M(t) =\; \sum_{m=0}^{M+1}\hat{u}_m\widetilde{T}_{m}(t),
\end{equation}
such that
\begin{equation}\label{ODETn}
\left\{\begin{aligned}
&\frac{d}{dt}u_M(t_m)=\;\mathcal{I}_{M}f(t_m,u_M(t_m)), \quad \quad m=0,1,\cdots,M,\\
&u_M(a) =\; u_a,
\end{aligned}\right.
\end{equation}
where $\mathcal{I}_{M}f\left(t,u(t)\right):C(a,b)\rightarrow\mathcal{P}_{M}(a,b)$
is the Chebyshev interpolation of $f\left(t,u(t)\right)$ defined by
\begin{equation}\begin{split}
\mathcal{I}_{M}f(t,u_M(t)) =\; \sum_{m=0}^{M}\hat{f}_m\widetilde{T}_{m}(t),
\end{split}\end{equation}
the coefficients $\{\hat{f}_m\}_{m=0}^{M}$ are determined by the forward discrete
Chebyshev transform
\begin{equation}\label{fk}\begin{split}
\hat{f}_m=\frac{2}{c_m(M+1)}\sum_{l=0}^Mf(t_l,u_M(t_l))\widetilde{T}_{m}(t_l),
\end{split}\end{equation}
where $c_0=2$, $c_m=1$ for $m\geq 1$.
Then by the shifted Chebyshev derivative relationship \eqref{Che_diff},
we can derive the coefficients $\{\hat{u}_m\}_{m=0}^{M+1}$ in \eqref{uM} by
\begin{equation}\label{uk}\begin{split}
&\hat{u}_{M+1} =\; \frac{(b-a)}{4(M+1)}\hat{f}_M,\quad \hat{u}_{M} =\; \frac{(b-a)}{4M}\hat{f}_{M-1},\\
&\hat{u}_{m} =\; \frac{b-a}{4m}(c_{m-1}\hat{f}_{m-1}-\hat{f}_{m+1}),\quad 1\leq m\leq M-1,\\
&\hat{u}_{0} =\; u_{a} - \sum_{k=1}^{M+1}(-1)^{k}\hat{u}_{k}.
\end{split}\end{equation}

Yang and Wang \cite{YW2015} presented the error estimate for the $hp$-version
of the single interval Chebyshev-Gauss spectral collocation method,
before introduce the theorem, we first present some notations used throughout the error estimate
\begin{itemize}
\item $H^r_{\omega}(a,b)$, $(r\geq 0)$ denotes the weighted Sobolev space
with certain weight function $\omega=(t-a)(b-t)$ in $(a,b)$,
especially, $H^0_{\omega}(a,b)=L^2_{\omega}(a,b)$.
\item $\|v\|_{\omega}$ denotes the norm of the space $L^2_{\omega^{-1/2}}(a,b)$.
\end{itemize}
The error estimate theorem is stated as follows.
\begin{thm}\label{theo_CGspec} 
Assume that $f(t,z)$ fulfills the Lipschitz condition,
\begin{equation}
|f(z_1,t)-f(z_2,t)|\leq L|z_1-z_2|, \quad L>0,
\end{equation}
and $0<L (b-a)<\beta<\frac{1}{4}$ ($\beta$ is a certain constant) holds.
Then for any $u\in H^r_{\omega^{r-\frac{3}{2}}}(a,b)$ with integers $2\leq r\leq N+1$, we have
\begin{equation}\begin{split}
\|u-u_M\|^2_{L^2(a,b)}\leq&\; \frac{b-a}{2} \|u-u_M\|^2_{\omega}
          \leq C_{\beta}(b-a)^3M^{4-2r}\int_{a}^{b}\omega^{r-\frac{3}{2}}(t)\left(\frac{d^r}{dt^r}u(t)\right)^2dt,\\
|u(b)-u_M(b)|\leq&\; C_\beta (b-a)^2M^{4-2r}\int_{a}^{b}\omega^{r-\frac{3}{2}}(t)\left(\frac{d^r}{dt^r}u(t)\right)^2dt.
\end{split}\end{equation}
where $C_{\beta}$ is a positive constant depending only on $\beta$.
\end{thm}

In actual computation, we use Picard iteration procedure to compute the coefficients $\{\hat{u}_m\}_{m=0}^M$.
The approach is simple to implement, especially for the complex nonlinear problems.
The $p$-th ($p=1,2,\cdots$) Picard iteration form of \eqref{ODETn} is
\begin{equation}
\frac{d}{dt}u^{p+1}_M(t_m)=\;\mathcal{I}_{M}f\left(t_m,u^p_M(t_m)\right),
\end{equation}
if the equation satisfies the convergence condition in Theorem \ref{theo_CGspec},
the iteration solution $u_M^p(t)$
will converge to the numerical solution $u_M(t)$ with large enough $p$,
and the convergence is of order one, that is to say there exists a constant $0<C_p<1$ such that
$\|e_{p+1}\|_{\infty}\leq C_p\|e_{p}\|_{\infty}$ with the definition $e_{p}:=u^M_p(t)-u^M(t)$.
\begin{algorithm}[htb]
\caption{Chebyshev-Gauss Spectral Collocation Algorithm}\label{Alg:CGspec}
\begin{algorithmic}
\STATE\textbf{Input}: Provide the initial guess of $\{u^M_0(t_m)\}_{m=0}^M$, the tolerance $\epsilon$.
\STATE \textbf{For} $p = 0, 1,\cdots$
\STATE\textbf{Step 1}. Evaluate the values of $\{f(t_m,u_p^M(t_m))\}_{m=0}^M$,
\STATE\textbf{Step 2}. Compute the coefficients $\{\hat{f}_{m}^p\}_{m=0}^M$ by \eqref{fk}.
\STATE\textbf{Step 3}. Compute the coefficients $\{\hat{u}_{m}^p\}_{m=0}^{M+1}$ by \eqref{uk}.
\STATE\textbf{Step 4}. Update the data of $\{u^{p+1}_M(t_m\}_{m=0}^M$ by \eqref{uM}.
\STATE\textbf{Step 5}. If the iteration error satisfies the stopping criterion,
$$\|u^{p+1}_M(t)-u^{p}_M(t)\|_\infty<\epsilon,$$
\quad\quad\quad\quad terminate the iteration; otherwise go back to Step 1.
\STATE\textbf{Step 6}. Compute $u_M^{p+1}(b)=\sum_{m=0}^{M+1}\hat{u}_{m}^p$.
\end{algorithmic}
\end{algorithm}

We designate $\mathcal{F}_{\rm CG}$ as the numerical propagator defined
by the Chebyshev-Gauss spectral collocation method at the conclusion of the section.
The numerical output of Algorithm \ref{Alg:CGspec} with $M$ points in the interval $[a,b]$
is represented as $\mathcal{F}_{\rm CG}(a,u_a,M,b-a)$. That is,
\begin{equation}\label{F_CG}
u_M^{p+1}(b)=\mathcal{F}_{\rm CG}(a,u_a,M,b-a).
\end{equation}
\section{Parareal-CG algorithm}\label{sec:para-CG}
The parareal algorithm introduced by Gander and Vandewalle \cite{GV2007} is revisited in this section,
using the Chebyshev-Gauss spectral collocation method as the fine propagator.
First, we divide the whole time interval $[0,T]$ uniformly 
by $0=T_0<T_1<\cdots<T_N=T$ and define $\Delta T=T/N$.
Second, we divide each $(T_n,T_{n+1})$ by $M$($\geq 2$) shifted Chebyshev-Gauss points \eqref{shifted_CG}.
Then, the low-order and inexpensive numerical method $\mathcal{G}$ propagator
is applied to the coarse time grids,
while the Chebyshev-Gauss spectral collocation method $\mathcal{F_{\rm CG}}$ propagator
having spectral accuracy is utilized in the fine grids.
The time-sequential and time-parallel parts of the algorithm
are denoted by the symbols $\ominus$ and $\oplus$, respectively.
The following is the parareal method employing $\mathcal{F_{\rm CG}}$ as the fine propagator.
\begin{algorithm}[h]
\caption{Parareal-CG Algorithm}\label{Alg:paraCG}
\begin{algorithmic}
\STATE $\ominus$ \textbf{Initialization}: Compute sequentially $u_{n+1}^{0}=\mathcal{G}(T_n,u_n^0,\triangle T)$ with $u_0^0=u_0$, $n=0,1,\cdots,N-1$;
\STATE \textbf{For} $k = 0, 1,\cdots$
\STATE $\oplus$ \textbf{Step 1}. On each subinterval $[T_n,T_{n+1}]$, compute
$\tilde{u}_{n+1}=\mathcal{F_{\rm CG}}(T_{n},u^k_n,M,\Delta T)$. 
\STATE $\ominus$ \textbf{Step 2}. Perform sequential corrections
\begin{equation}
u_{n+1}^{k+1}=\mathcal{G}(T_n,u_n^{k+1},\triangle T)+\tilde{u}_{n+1}-\mathcal{G}(T_n,u_n^{k},\triangle T),
\end{equation}
where $u^{k+1}_0 = u_0$, $n = 0, 1,\cdots,N-1$;
\STATE $\ominus$ \textbf{Step 3}. If $\{u^{k+1}_n\}^N_{n=1}$ satisfies the stopping criterion,
terminate the iteration and output $\{u^{k+1}_n\}^N_{n=1}$; otherwise go back to Step 1.
\end{algorithmic}
\end{algorithm}

Given that the implicit Euler method is $L$-stable
and feasible for high coarse time step sizes $\Delta T$
forced by the parareal algorithm, using it as the coarse propagator makes sense.
Other reliable implicit numerical methods, such implicit Runge-Kutta methods,
are also available for use as the $\mathcal{G}$ propagator,
although they are all significantly more costly than the implicit Euler method.
In this paper, we apply the Chebyshev-Gauss spectral collocation method to the fine propagator,
and the compact form of the Parareal-CG method may be expressed as
\begin{equation}
u_{n+1}^{k+1}=\mathcal{G}(T_n,u_n^{k+1},\triangle T)+\mathcal{F_{\rm CG}}(T_{n},u^k_n,M,\Delta T)-\mathcal{G}(T_n,u_n^{k},\triangle T).
\end{equation}

For comparison, we also take into account other high-order and expensive numerical methods like trapezoidal formula.
Accordingly, each interval $[T_n,T_{n+1}]$ should be divided into $J$ ($\geq 2$) small time-intervals
$[T_{n+\frac{j}{J}},T_{n+\frac{j+1}{J}}]$, $j=0,1,\cdots, J-1$.
We assume the intervals are of uniform size, therefore, $\Delta t=\frac{\Delta T}{J}$.
After that, the parareal algorithm can be derived by
\begin{algorithm}[h]
\caption{Parareal Algorithm}\label{Alg:para}
\begin{algorithmic}
\STATE $\ominus$ \textbf{Initialization}: Compute sequentially $u_{n+1}^{0}=\mathcal{G}(T_n,u_n^0,\triangle T)$ with $u_0^0=u_0$, $n=0,1,\cdots,N-1$;
\STATE \textbf{For} $k = 0, 1,\cdots$
\STATE $\oplus$ \textbf{Step 1}. On each subinterval $[T_n,T_{n+1}]$, compute
$\tilde{u}_{n+\frac{j+1}{J}}=\mathcal{F}(T_{n+\frac{j}{J}},\tilde{u}_{n+\frac{j}{J}},\frac{\Delta T}{J})$
with  initial value $\tilde{u}_n = u_n^k$, where $T_{n+\frac{j}{J}}=T_n+\frac{j\Delta T}{J}$ and $j=0,1,\cdots,J-1$;
\STATE $\ominus$ \textbf{Step 2}. Perform sequential corrections
\begin{equation}
u_{n+1}^{k+1}=\mathcal{G}(T_n,u_n^{k+1},\triangle T)+\tilde{u}_{n+1}-\mathcal{G}(T_n,u_n^{k},\triangle T),
\end{equation}
where $u^{k+1}_0 = u_0$, $n = 0, 1,\cdots,N-1$;
\STATE $\ominus$ \textbf{Step 3}. If $\{u^{k+1}_n\}^N_{n=1}$ satisfies the stopping criterion,
terminate the iteration and output $\{u^{k+1}_n\}^N_{n=1}$; otherwise go back to Step 1.
\end{algorithmic}
\end{algorithm}

The compact form of the parareal algorithm can be derived by
\begin{equation}
u_{n+1}^{k+1}=\mathcal{G}(T_n,u_n^{k+1},\triangle T)+\mathcal{F}^J(T_n,u_n^k,\triangle t)-\mathcal{G}(T_n,u_n^{k},\triangle T),
\end{equation}
where $\mathcal{F}^J(T_n,u_n^k,\triangle t)$ stands for
the result of running $J$ steps of the fine propagator $\mathcal{F}$ with initial value $u_n^k$
and the small step-size $\Delta t$.
\section{Convergence Analysis}\label{sec:con}
Based on the parareal convergence analysis given by Gander and Vandewalle \cite{GV2007},
the convergence of Parareal-CG method is analysed
by constructing the stability function of Chebyshev-Gauss spectral collocation method
in this section.
\subsection{Stability function of Chevyshev-Gauss Spectral method}
In order to obtain the convergence factor $\mathcal{K}(z,M)$ of Parareal-CG algorithm in Algorithm \ref{Alg:paraCG},
in this subsection we present the stability function of Chebyshev-Gauss spectral collocation method $\mathcal{R}_{\rm CG}(z,M)$.
\begin{lem}\label{RMCPIlem}
For given $M$ and $z:=\lambda\triangle T$, the stability function $\mathcal{R}_{\rm CG}(z,M)$
of Chebyshev-Gauss spectral collocation method in $[T_n,T_{n+1}]$, ($n=1,2,\cdots, N$) is
\begin{equation}\label{RMCPI}
\mathcal{R}_{\rm CG}(z,M)=\boldsymbol{T}(\boldsymbol{I_2}-z\boldsymbol{C_\alpha}
(\boldsymbol{I_1}+z\boldsymbol{T_1}\boldsymbol{C_{\alpha}})^{-1}\boldsymbol{T_1})\boldsymbol{E},
\end{equation}
where $\boldsymbol{T_2}$ and $\boldsymbol{C_{\alpha}}$ are the coefficient matrices
defined in \eqref{mar_f} and \eqref{sch_mv}, respectively;
$\boldsymbol{I_1}$ and $\boldsymbol{I_2}$ are two identity matrices of the size $(M+1)$ and $(M+2)$, respectively;
$\boldsymbol{T}=[1,1,\cdots,1]_{_{1\times(M+2)}}$ and $\boldsymbol{E}=[1,0,\cdots,0]^{\top}_{_{(M+2)\times1}}$
are two constant vectors.
\end{lem}
\begin{proof}
We first introduce the matrix-vector form of the Chebyshev-Gauss spectral collocation method
to be able to express $\mathcal{R}_{\rm CG}$ explicitly.
Denote the evaluations of the approximated polynomial for the $p$-th ($p=1,2,\cdots$) iteration $u^p_M(t)$
at the shifted Chebyshev-Gauss points $\{t_m\}_{m=0}^M$, ($n=0,1,\cdots,N-1$) by
\begin{equation}
\boldsymbol{u_p}=\;[u_p^M(t_0),u_p^M(t_1),\cdots,u_p^M(t_M)]^{\top}_{(M+1)\times 1},
\end{equation}
Equation \eqref{uk} indicates that the vector $\boldsymbol{u_p}$
can be expressed as a consequence of the fact $\mathcal{T}_l(t_m)=T_l(\tau_m)$, ($l=1,2,\cdots$; $m=0,1,\cdots,M$)
\begin{equation}
\boldsymbol{u_p}=\;\boldsymbol{T_1\hat{u}_p},
\end{equation}
where $\boldsymbol{T_1}$ is a coefficient matrix defined by
\begin{equation}\label{T1}\begin{split}
\small{\boldsymbol{T_1}=\; \mathop{\renewcommand\arraystretch{1.5}
\begin{bmatrix}
T_0(\tau_0)  &T_1(\tau_0) & \cdots & T_{M+1}(\tau_0)  \\
T_0(\tau_1)  &T_1(\tau_1) & \cdots & T_{M+1}(\tau_1)  \\
\vdots  &\vdots & \ddots  &\vdots   \\
T_0(\tau_M)  &T_1(\tau_M) & \cdots & T_{M+1}(\tau_M)  \\
\end{bmatrix}}\limits^{(M+1)\times(M+2)}}.
\end{split}\end{equation}
Denote $\hat{f}(t_m,u^p_M(t_m))=\hat{f}_{m}^p$ and suppose the initial value be $u(T_n)=u_{T_n}$,
then the coefficient vector $\boldsymbol{\hat{u}^p}$ can be expressed by
\begin{small}
\begin{equation}\begin{split}
\boldsymbol{\hat{u}^p}=\;&[\hat{u}_{0}^p,\hat{u}_{1}^p,\cdots,\hat{u}_{M+1}^p]^{\top}_{_{(M+2)\times 1}}\\
=\;&\left[u_{T_n} + \sum_{m=1}^{M+1}(-1)^{m-1}\hat{u}_{m}^p,
\frac{\Delta T}{4} (2\hat{f}_{0}^{p-1}-\hat{f}_{2}^{p-1}),
\cdots,
\frac{\Delta T(\hat{f}_{M-2}^{p-1}-\hat{f}_{M}^{p-1})}{4(M-1)} ,
\frac{\Delta T}{4M} \hat{f}_{M-1}^{p-1},
\frac{\Delta T}{4(M+1)} \hat{f}_{M}^{p-1}\right]^{\top}\\
=\;&
\mathop{\renewcommand\arraystretch{1.5}
\begin{bmatrix}
u_{T_n}\\ 0\\ 0\\ \vdots\\ 0 \\0 \\0
\end{bmatrix}}
+\frac{\Delta T}{4}
\mathop{\renewcommand\arraystretch{1.5}
\begin{bmatrix}
1  & & & & & & \\
  & 1 & & & & &\\
  & &\frac{1}{2} & & & &\\
  & & & \ddots & & &\\
       & & & & \frac{1}{M-1} & &\\
     & & & & & \frac{1}{M} &\\
  & & & & & & \frac{1}{M+1}\\
\end{bmatrix}}\limits^{\boldsymbol{R}:(M+2)\times(M+2)}
\mathop{\renewcommand\arraystretch{1.5}
\begin{bmatrix}
2  &-\frac{1}{2} & s_2 & s_3 &  \cdots  & s_{M}  \\
2  & 0     & -1   & 0       & \cdots   & 0  \\
0   & 1    & 0     &-1       & \cdots   & 0   \\
\vdots &\ddots &\ddots &\ddots &\ddots  & \vdots  \\
0 &  0       & \ddots      & 1   & 0 & -1 \\
0 &  0       & 0     & \ddots   & 1  & 0    \\
0 &  0       & 0     & \cdots   & 0  & 1   \\
\end{bmatrix}}\limits^{\boldsymbol{S}:(M+2)\times(M+1)}
\mathop{\renewcommand\arraystretch{1.5}
\begin{bmatrix}
\hat{f}_{0}^{p-1}\\ \hat{f}_{1}^{p-1}\\f_{2}^{p-1}
\\ \hat{f}_{3}^{p-1}\\ \vdots \\ \hat{f}_{M}^{p-1}
\end{bmatrix}}\\
=\;&\boldsymbol{U_0}+\frac{\Delta T}{4}\boldsymbol{RS\hat{f}^p},
\end{split}\end{equation}
\end{small}
where $\boldsymbol{R}$ and $\boldsymbol{S}$ are the coefficient matrices and
the first line of $\boldsymbol{S}$ satisfies
$$s_m=\;(-1)^{m}\left(\frac{1}{m+1}-\frac{1}{m-1}\right),\quad m=2,3,\cdots,M.$$
$\boldsymbol{U_0}$ and $\boldsymbol{\hat{f}^p}$ can be defined by
\begin{equation}\begin{split}
&\boldsymbol{U_0}=\;[u_{T_n},0,0,\cdots,0]^{\top}_{_{(M+2)\times 1}},\\
&\boldsymbol{\hat{f}^p}=\;[\hat{f}_{0}^{p-1}, \hat{f}_{1}^{p-1},\cdots,\hat{f}_{M}^{p-1}]^{\top}_{_{(M+1)\times 1}}.
\end{split}\end{equation}
For simplicity, we set the values of the function $f(t,u^{p-1}(t))$ on the Chebyshev-Gauss points by the notation
$f_{m}^{p-1}=f\left(t_m,u_{p-1}^M(t_m)\right)$, $m=0,1,\cdots,M$,
then \eqref{fk} can be expressed as
\begin{equation}\label{mar_f}
\begin{split}
\boldsymbol{\hat{f}^{p-1}}=&\;
\mathop{\renewcommand\arraystretch{1.5}
\begin{bmatrix}
\frac{1}{M+1}(f_{0}^{p-1}T_0(\tau_0)+f_{1}^{p-1}T_0(\tau_1)+\cdots+f_{M}^{p-1}T_0(\tau_M))\\
\frac{2}{M+1}(f_{0}^{p-1}T_1(\tau_0)+f_{1}^{p-1}T_1(\tau_1)+\cdots+f_{M}^{p-1}T_1(\tau_M))\\
\vdots\\
\frac{2}{M+1}(f_{0}^{p-1}T_M(\tau_0)+f_{1}^{p-1}T_M(\tau_1)+\cdots+f_{M}^{p-1}T_M(\tau_M))\\
\end{bmatrix}}\\
=&\;\mathop{\renewcommand\arraystretch{1.5}
\begin{bmatrix}
\frac{1}{M+1}  & & &\\
 & \frac{2}{M+1}& & \\
 & &\ddots & \\
 & & & \frac{2}{M+1}\\
\end{bmatrix}}\limits^{\boldsymbol{V}:(M+1)\times(M+1)}
\mathop{\renewcommand\arraystretch{1.5}
\begin{bmatrix}
T_0(\tau_0)  &T_0(\tau_1) & \cdots & T_0(\tau_M)  \\
T_1(\tau_0)  &T_1(\tau_1) & \cdots & T_1(\tau_M)  \\
\vdots  &\vdots & \ddots  &\vdots   \\
T_M(\tau_0)  &T_M(\tau_1) & \cdots & T_{M}(\tau_M)  \\
\end{bmatrix}}\limits^{\boldsymbol{T_2}:(M+1)\times(M+1)}
\mathop{\renewcommand\arraystretch{1.5}
\begin{bmatrix}
f_{0}^{p-1}\\ f_{1}^{p-1}\\ \vdots \\ f_{M}^{p-1}
\end{bmatrix}}\\
=&\;\boldsymbol{VT_2f^{p-1}}.
    \end{split}
\end{equation}
where $\boldsymbol{V}$ and $\boldsymbol{T_2}$ are coefficient matrices
and the vector $\boldsymbol{f^{p-1}}$ is defined by
\begin{equation}\begin{split}
\boldsymbol{f^{p-1}}=\;[f(t_0,u^{p-1}_M(t_0)),
f(t_1,u^{p-1}_M(t_1)),\cdots,f(t_M,u^{p-1}_M(t_M))]^{\top}_{_{(M+1)\times 1}}.
\end{split}\end{equation}
To this end, given the initial guess $\boldsymbol{u_0}$,
the $p$-th matrix-vector iteration version of the Chebyshev-Gauss spectral collocation method
for solving \eqref{ODETn} is as follows.
\begin{equation}\label{sch_mv}\begin{split}
&\boldsymbol{C_{\alpha}}=\;\frac{1}{4}\boldsymbol{RSVT_2},\\
&\boldsymbol{\hat{u}^p} =\; \boldsymbol{U_0}-\Delta T\boldsymbol{C_{\alpha}f^{p-1}},\\
&\boldsymbol{u^{p}}=\;\boldsymbol{T_1\hat{u}^p}.
\end{split}\end{equation}
We develop the matrix-vector form of the method in order to acquire the stability function on the one hand,
and to considerably improve computing efficiency so that it may be employed in our numerical experiments on the other.
The matrix-vector form of the method improves the efficiency of computation sufficiently.
Following, we establish the linear stability function for the Chebyshev-Gauss spectral collocation method,
for which we initially derive the special form of the method for solving the linear equation,
\begin{equation}\label{Uk1}\begin{split}
\boldsymbol{u^{p}}=\;\boldsymbol{T_1\hat{u}^p}
=&\;\boldsymbol{T_1}(\boldsymbol{U_0}-\Delta T\boldsymbol{C_{\alpha}f^{p-1}})\\
=&\;\boldsymbol{T_1U_0}-\lambda\Delta T\boldsymbol{T_1C_\alpha \boldsymbol{u^{p-1}}}.\\
\end{split}\end{equation}
By assuming $\boldsymbol{u}^{*}=[u^*(t_0),u^*(t_1),\cdots,u^*(t_M)]^{\top}$
be the convergence solution after a sufficient number of iteration,
we may rewrite \eqref{Uk1} for given $z:=\lambda \Delta T$ as
\begin{equation}\label{U*}\begin{split}
\boldsymbol{u^{*}}
=&\;\boldsymbol{T_1U_0}-z\boldsymbol{T_1C_\alpha \boldsymbol{u^{*}}},\\
=&\;(\boldsymbol{I_1}+z\boldsymbol{T_1C_\alpha})^{-1}\boldsymbol{T_1U_0},
\end{split}\end{equation}
where $\boldsymbol{I_1}$ is an identity matrix of size ($M+1$),
then we can derive the corresponding coefficient vector $\boldsymbol{\hat{u}}^{*}=[u^*_1,u^*_2,\cdots,u^*_{M+1}]^{\top}$
directly  by \eqref{sch_mv}
\begin{equation}\begin{split}
\boldsymbol{\hat{u}^*}=\;\boldsymbol{U_0}-\Delta T\boldsymbol{C_{\alpha}}(\lambda \boldsymbol{u}^*)
=\boldsymbol{U_0}-z\boldsymbol{C_{\alpha}}(\boldsymbol{I_1}+z\boldsymbol{T_1C_\alpha})^{-1}\boldsymbol{T_1U_0}.
\end{split}\end{equation}
Finally, the stability function of Chebyshev-Gauss spectral collocation method $\mathcal{R}_{\rm CG}(z,M)$
could be acquired by the compression relations
\begin{equation}\label{R_CG}
\begin{split}
\mathcal{R}_{\rm CG}(z,M)=\frac{u^*(T_{n+1})}{u_{T_n}}=\frac{\boldsymbol{T}\boldsymbol{\hat{u}^*}}{u_{T_n}}
=\boldsymbol{T}(\boldsymbol{I_2}-z\boldsymbol{C_\alpha}(\boldsymbol{I_1}+z\boldsymbol{T_1C_{\alpha}})^{-1}\boldsymbol{T_1})\boldsymbol{E},
\end{split}
\end{equation}
where $\boldsymbol{I_2}$ is an $(M+2)$ dimensional identity matrix, $\boldsymbol{T}$ and $\boldsymbol{E_1}$
are two vectors defined by $\boldsymbol{T}=[1,\cdots,1,1]_{_{1\times(M+2)}}$ and
$\boldsymbol{E}=[1,0,\cdots,0]_{_{(M+2)\times1}}^{\top}$.
\end{proof}

We can obtain the stability functions $\mathcal{R}_{\rm CG}$,
particularly when $M=0$ and $M=1$.
\begin{equation}
\mathcal{R}_{\rm CG}(z,0)=\frac{2-z}{2+z},\quad
\mathcal{R}_{\rm CG}(z,1)=\frac{z^2-8z+16}{z^2+8z+16}.
\end{equation}
For $M=0, 1, 2, 4, 20$, we displayed $\left|\mathcal{R}_{\rm CG}(z,M)\right|$
as the functions of $z_{\max}:=\triangle T\lambda_{\max}$ in Figure \ref{fig:R_CG},
which shows that for various $M$ we have
\begin{equation}\begin{split}
\lim_{z\rightarrow\infty}\left|\mathcal{R}_{\rm CG}(z)\right|=\; 1.
\end{split}\end{equation}
\begin{figure}
\centerline{\includegraphics[width=0.5\textwidth]{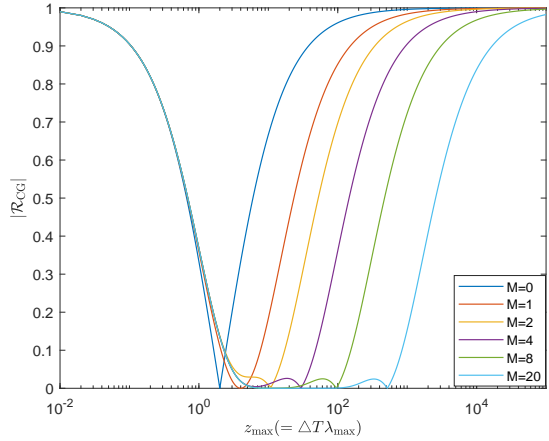}}
\caption{The contraction factor $|\mathcal{R}_{\rm CG}|$ as a function of $z_{\max}$, for different $M$.}
\label{fig:R_CG}
\end{figure}

As a result, we can obtain the contraction factor of Parareal-CG method and derive the convergence analysis.
\subsection{Convergence analysis of Parareal-CG method}
The following conclusion regarding the convergence of Parareal-CG Algorithm
on sufficiently long time intervals 
can be deduced directly from the analysis provided by Gander and Vandewalle \cite{GV2007}
for the linear system of ODEs \eqref{SPD} with $A\in\mathbb{R}^{m\times m}$ and symmetric positive definite matrix
(which therefore can be diagonalized and all the eigenvalues are positive real numbers).
\begin{thm}[\cite{GV2007}]\label{Para-CGthm}
Let $\mathcal{F}_{\rm CG}$ be the Chebyshev-Gauss spectral collocation propagator
with the stability function $\mathcal{R}_{\rm CG}(z)$ \eqref{R_CG},
and let $\sigma(A) = \{\lambda_1,\cdots, \lambda_m\}$ be the set of eigenvalues of the matrix $A$ in \eqref{SPD}.
Then, the errors $\{e^k_n\}$ of the parareal-CG algorithm at $k$-th iteration using the
backward-Euler method as the coarse propagator satisfy
\begin{equation}\label{rho}
\sup_n\|Ve^k_n\|_{\infty}\leq\rho^k\sup_n\|Ve^k_0\|_{\infty},\quad \rho=\max_{\lambda\in\sigma(A)}\mathcal{K}(\triangle T\lambda,M),
\end{equation}
where $\rho$ is the convergence factor of the parareal algorithm, $k \geq 1$ is the iteration index,
and $V\in \mathbb{R}^{m\times m}$ consists of the eigenvectors of $A$ $(i.e., V^{-1}AV = diag(\lambda_1,\cdots,\lambda_m))$.
The argument $\mathcal{K}_{\rm CG}$, which is the convergence factor corresponding to a single eigenvalue
(or in short "contraction factor" hereafter), is defined by
\begin{equation}\label{K_CG}
\mathcal{K}_{\rm CG}(z,J)=\frac{\left|\mathcal{R}_{\rm CG}(z,M)-\frac{1}{1+z}\right|}{1-\left|\frac{1}{1+z}\right|}.
\end{equation}
\end{thm}

The convergence theorem \ref{Para-CGthm} states that the behavior of the Parareal-CG algorithm
over a long time interval is determined by the convergence factor
$\rho(M)$ defined by $\rho(M):=\max_{z\in[0,z_{\max}]}\mathcal{K}(z,M)$
where $z_{\max} = \lambda_{\max}\Delta T$ and
the quantity $\lambda_{\max}$ denotes the maximal eigenvalue (or an upper bound)
of the coefficient matrix $A$ in \eqref{SPD}.
We can infer from the equation \eqref{rho} that the smaller $\rho(M)$ is, the faster the algorithm converges.
To maintain the convergence rate, the convergence factor $\rho(M)$ prefers to be around $\frac{1}{3}$.
\begin{figure}
\centerline{\includegraphics[width=0.5\textwidth]{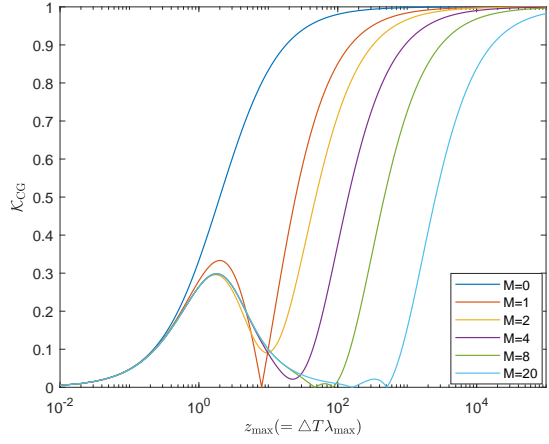}}
\caption{The contraction factor $\mathcal{K}_{\rm CG}$ as a function of $z_{\max}$, for different $M$.}
\label{fig:Kap}
\end{figure}

The convergence factor of Parareal-CG algorithm
$\mathcal{K}_{\rm CG}(z,M)$ as functions of
$z_{\max}:=\triangle T\lambda_{\max}$ for various $M$ are shown in Figure \ref{fig:Kap}.
We may infer from the behavior of the $\mathcal{K}_{\rm CG}(z,M)$
which is similar to the convergence factor of Parareal-TR and Parareal-Gauss 4 algorithms
discussed in \cite{WZ2015} that
\begin{equation}\begin{split}
\lim_{z\rightarrow 0}\mathcal{K}_{\rm CG}(z)=\; 0,\quad\quad
\lim_{z\rightarrow\infty}\mathcal{K}_{\rm CG}(z)=\; 1.
\end{split}\end{equation}
It implies that the Parareal-CG algorithm cannot keep the convergence rate for arbitrarily large $z$,
however, by solving an appropriately designed $M$ that is large enough,
it is still possible to get a feasible parareal solver.
In further detail, as $M$ grows, the convergence factor $\rho(M):=\max_{z\in[0,z_{max}]}\mathcal{K}_{\rm CG}(z,M)$
tends to be around $\frac{1}{3}$ for some given $z$.

The Convergence analysis can be derived by the convergence theorem \ref{Para-CGthm} of Parareal-CG algorithm
and the stability function \eqref{RMCPI} of Chebyshev-Gauss spectral collocation method.
\begin{thm}\label{thm_M}
Given a fixed constant $z_{\max}=\triangle T\lambda_{\max}>0$,
there exists some positive integer $M^*_{\min}$ such that.
\begin{equation}
\max_{z\in[0,z_{\max}]}\mathcal{K}_{\rm CG}(z,M)\leq \frac{1}{3},\quad  \text{if}\;\;\; M\geq M^*_{\min},
\end{equation}
where $\mathcal{K}_{\rm CG}$ is the contraction factors of the Parareal-CG algorithm defined by \eqref{K_CG}
with the stability functions of Chebyshev-Gauss spectral collocation method $\mathcal{R}_{\rm CG}$ given in \eqref{R_CG}.
\end{thm}
The Parareal-CG algorithm's lower bound, $M^*_{\min}$, is provided by
\begin{equation}
M^*_{\min}=\left\{\begin{aligned}
&0\quad \quad\quad \text{if}\;\; z\leq z_0^*, \\
&1\quad \quad\quad \text{if}\;\; z_0^*<z\leq z^*_1, \\
&M^*_{\rm CG} \quad\; \text{otherwise}.
\end{aligned}\right.
\end{equation}
where $M^*_{\rm CG}>1$ depends on $z_{\max}$ and is the minimum positive integer
which satisfies the following equation
\begin{equation}\label{root}
\left|\mathcal{R}_{\rm CG}(z_{\max},M)\right|\leq \frac{3+z_{\max}}{3(1+z_{\max})}.
\end{equation}
Moreover, the quantity $z_0^*$ is the unique positive root of $\mathcal{K}_{\rm CG}(z,0)=\frac{1}{3}$
and $z_1^*$ is the maximum positive root of $\mathcal{K}_{\rm CG}(z,1)=\frac{1}{3}$,
that is to say, $z_0^*=1$ and $z_1^*=8+6\sqrt{2}$.

\begin{proof}
When $M=0$, we can derive the $\mathcal{R}_{\rm CG}(z,0)$ and
$\mathcal{K}_{\rm CG}(z,0)$ by
\begin{equation}
\mathcal{R}_{\rm CG}=\frac{2-z}{2+z},\quad\quad \mathcal{K}_{\rm CG}(z,0)=\frac{z}{2+z},\quad (z>0).
\end{equation}
It is obvious that $\mathcal{K}_{\rm CG}(z,0)$ is an increasing function with respect to $z>0$
and $\mathcal{K}_{\rm CG}(z,0)=\frac{1}{3}$ has the unique root $z0^*=1$.
that is to say,
\begin{equation}
\max_{z\in[0,z_{\max}]}\mathcal{K}_{\rm CG}(z,0)\leq\left\{\begin{aligned}
&\frac{1}{3},\quad\quad\quad\quad z\in(0,z_0^*]\\
&\frac{z_{\max}}{2+z_{\max}}\quad\; z\in(z_0^*,z_{\max}].
\end{aligned}\right.
\end{equation}
When $M=1$, we can derive the $\mathcal{R}_{\rm CG}(z,1)$ and
$\mathcal{K}_{\rm CG}(z,1)$ by
\begin{equation}
\mathcal{R}_{\rm CG}=\frac{(4-z)^2}{(4+z)^2},\quad\quad \mathcal{K}_{\rm CG}(z,1)=\frac{|z^2-8z|}{(4+z)^2},\quad z>0.
\end{equation}
There are three roots $z_{11}^*=z_{12}^*=2$ and $z_{1}^*=8+6\sqrt{2}$ of
$\mathcal{K}_{\rm CG}(z,1)=\frac{1}{3}$, respectively.
$z_{11}^*=z_{12}^*=2$ is also the maximum value point of $\mathcal{K}_{\rm CG}(z,1)$.
Moreover, $\mathcal{K}_{\rm CG}(z,1)$ is an increasing function for $z\geq 8+6\sqrt{2}$.
In conclusion, we may obtain
\begin{equation}
\max_{z\in[0,z_{\max}]}\mathcal{K}_{\rm CG}(z,0)\leq\left\{\begin{aligned}
&\frac{1}{3},\quad\quad\quad\quad\quad\quad z\in(0,z_1^*]\\
&\frac{z_{\max}^2-8z_{\max}}{(4+z_{\max})^2}\quad\; z\in(z_1^*,z_{\max}].
\end{aligned}\right.
\end{equation}
When $M\geq 2$, we can infer from Figure \ref{fig:Kap} that there exists only one root depends on $M$
$z^*(M)>z_1^*$ of $\mathcal{K}_{\rm CG}(z,M)=\frac{1}{3}$ and
$\mathcal{K}_{\rm CG}(z,M)$ is an increasing function for $z\geq z^*(M)$,
then we have
\begin{equation}
\max_{z\in[0,z_{\max}]}\mathcal{K}_{\rm CG}(z,M)=\left\{\begin{aligned}
&\frac{1}{3},\quad\quad\quad\quad\quad\quad\; z\in(0,z^*(M)]\\
&\mathcal{K}_{\rm CG}(z_{\max},M)\quad\, z\in(z^*(M),z_{\max}].
\end{aligned}\right.
\end{equation}
Therefore, we can derive that Theorem \ref{thm_M} holds.
\end{proof}

We can infer from Theorem \ref{thm_M} that $M$ has a significant impact
on the convergence rates because the lower bound $M^*_{\min}$ increases as $z_{\max}=\triangle T\lambda_{\max}$ grows.
Using \emph{while} loops in \emph{Matlab}, we can derive $M^*_{\min}$ for a given $z_{\max}$
since $M^*_{\min}$ is a natural number.
\begin{figure}[!ht]
\centerline{
\includegraphics[width=0.45\textwidth]{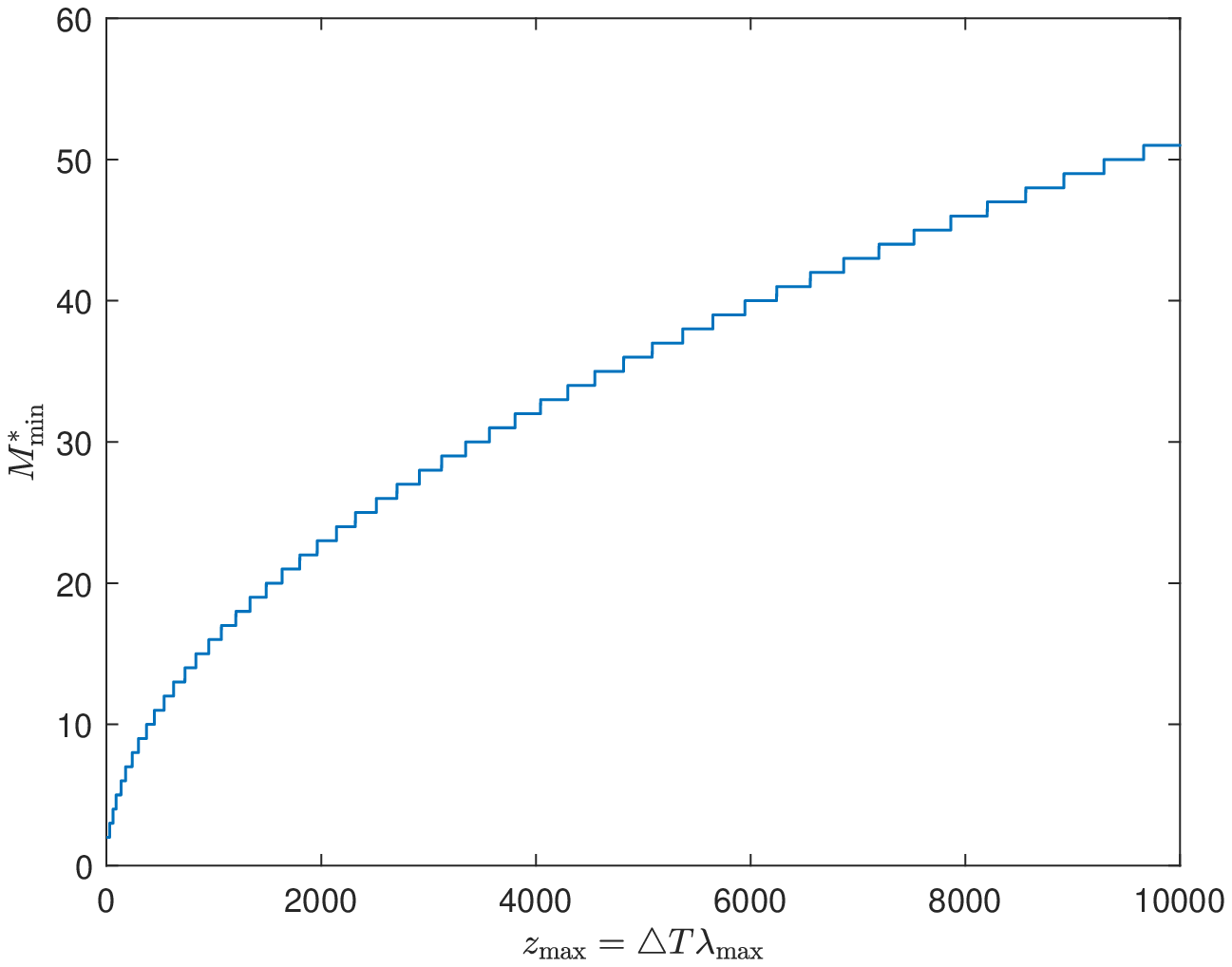}
\includegraphics[width=0.45\textwidth]{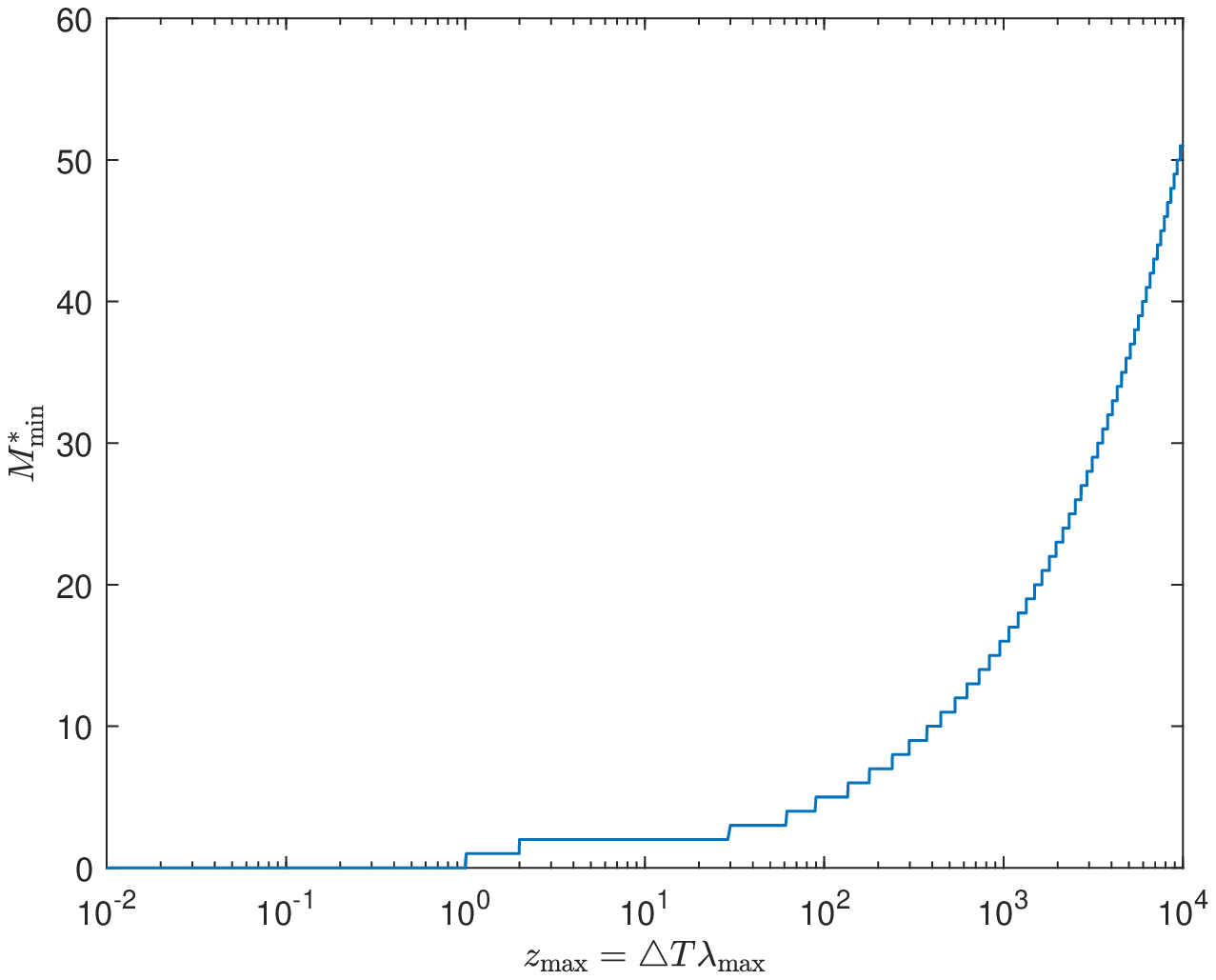}}
\caption{$M^*_{\min}$}
\label{M_CG}
\end{figure}
\begin{rem}[Comparison to other Parareal algorithms]\label{rem:com}
Researchers have proved that the following proposition for the Parareal-Euler \cite{MSS2010}
and Parareal-TR/BDF2 \cite{W2015} algorithms
utilizing the backward Euler and TR/BDF2 (i.e. ode23 solver for ODEs in MATLAB), respectively,
as the $\mathcal{F}$ propagators in the earlier papers.
\begin{equation}
\max_{z\in[0,z_{\max}]}\mathcal{K}_{\rm Euler,\;TR/BDF2}(z,J)<\frac{1}{3}, \quad \forall\, z_{\max}>0,
\end{equation}
which implies that for all $z_{\max}>0$, the convergence rates of Parareal algorithms $\rho(J)<\frac{1}{3}$.

Unfortunately, not all methods yield such a uniform result.
The following conclusion holds for the Parareal-TR and Parareal-Gauss4 algorithms \cite{WZ2015}
when utilizing the trapezoidal rule and fourth-order Gauss Runge-Kutta method as $\mathcal{F}$-propagator.
\begin{equation}
\max_{z\in[0,z_{\max}]}\mathcal{K}_{\rm TR,\; Gauss4}(z,J)<\frac{1}{3}, \quad \text{if $J$ is even and}\;\; J\geq J^*_{\min},
\end{equation}
It indicates that if the mesh ratio $J$ is large enough,
the two parareal algorithms will converge as fast as the Parareal-Euler method.
Our Parareal-CG method behaves similarly as Parareal-TR and Parareal-Gauss4 algorithms.

For the explicit $\mathcal{F}$-propagator forward Euler and
fourth-order explicit Runge-Kutta algorithms used in the Parareal-fEuler and Parareal-4ERK algorithms, respectively,
the convergence factor $\mathcal{K}$ begins to suddenly trend to infinity at some $z_{\max}$.
Therefore, these parareal algorithms can be used to solve a very limited number of equations.
\begin{equation}
\mathcal{K}_{\rm fEuler,\; 4ERK}(z,J)=\infty,\quad z>z^*.
\end{equation}
The convergence factors of the three kinds of parareal algorithms are shown in Figure \ref{Fig:com},

\begin{figure}[!ht]
\centerline{\includegraphics[width=0.4\textwidth]{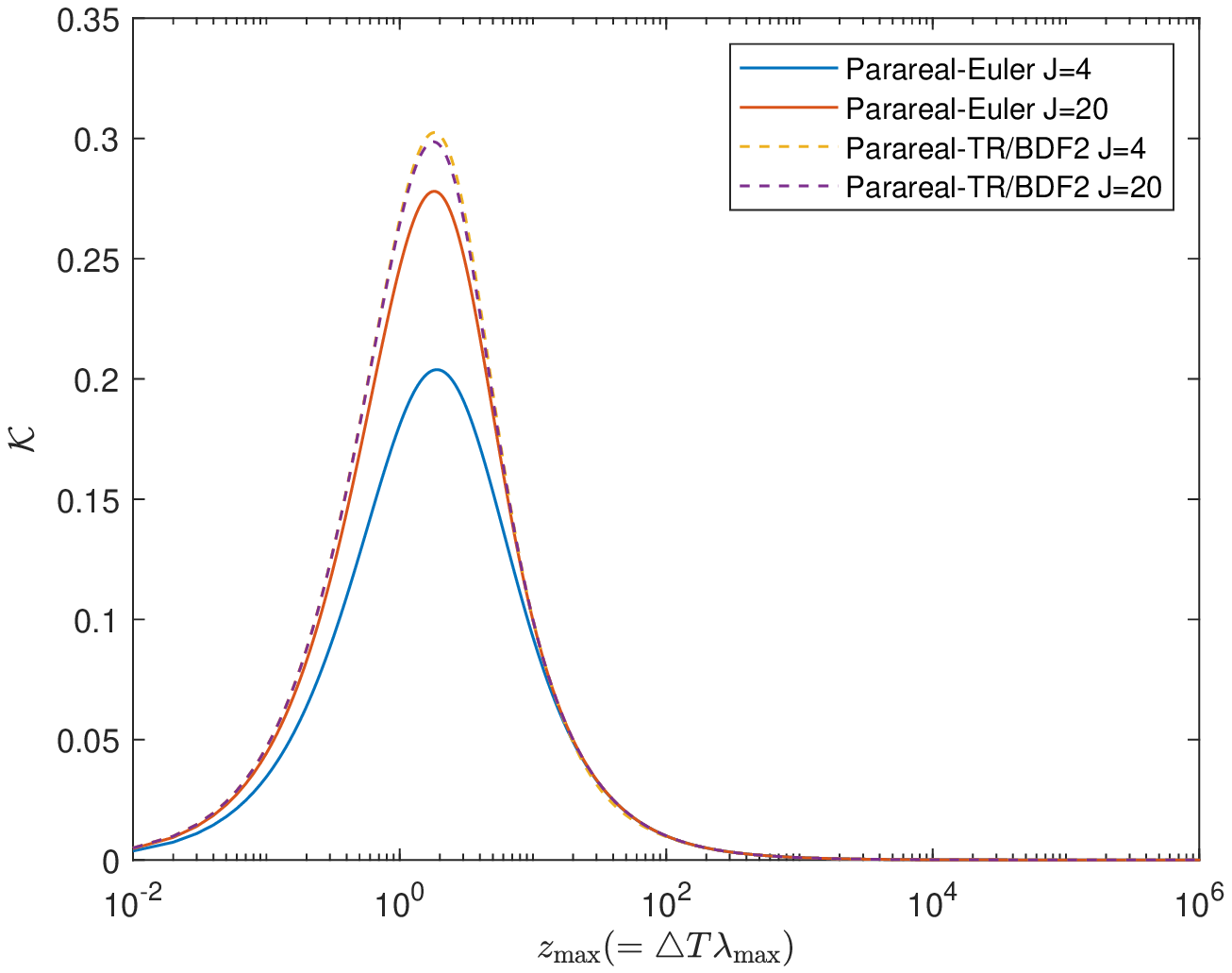}}
\centerline{\includegraphics[width=0.4\textwidth]{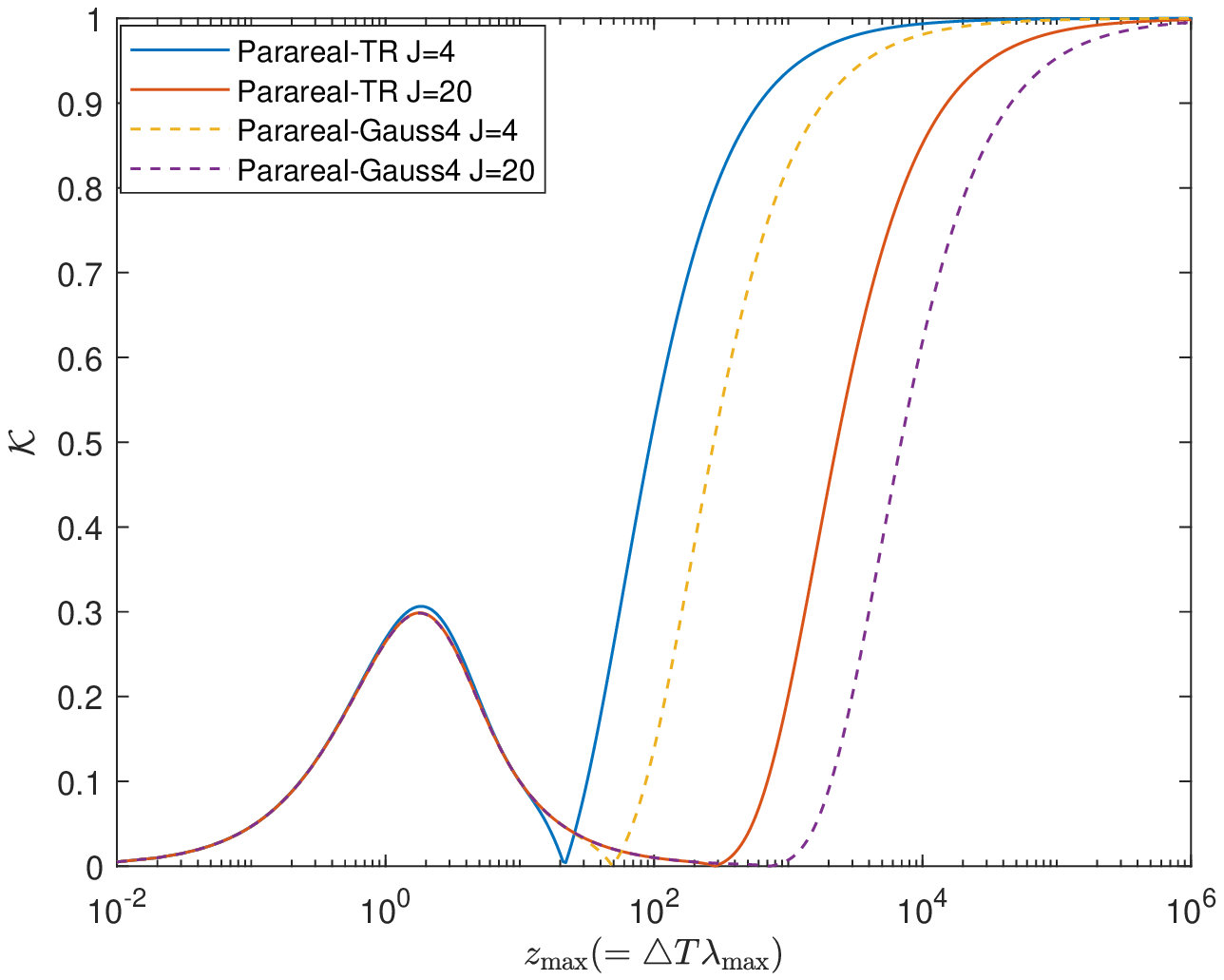}
\includegraphics[width=0.4\textwidth]{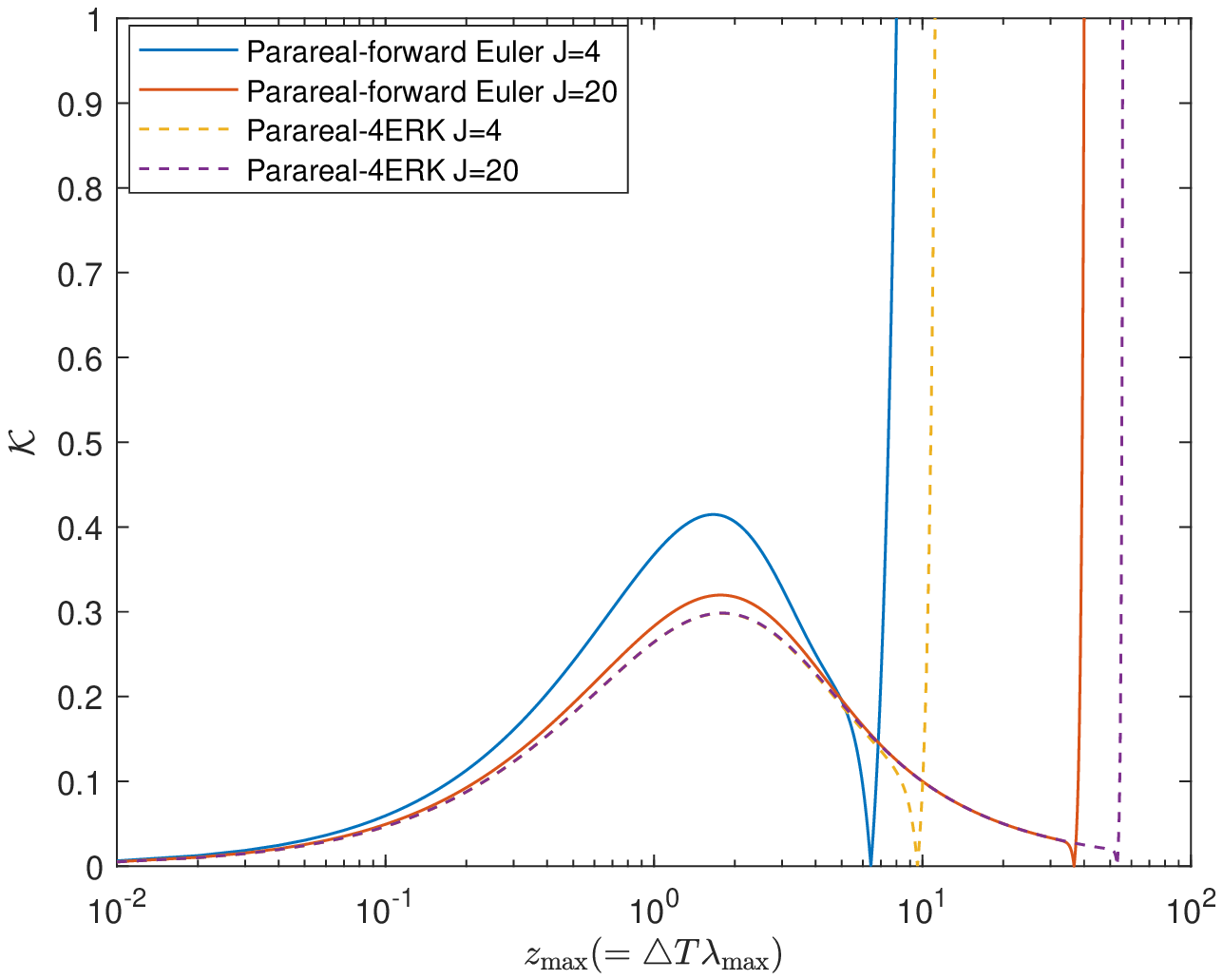}}
\caption{The contraction factors $\mathcal{K}$ as a functions of $z_{\max}$, for different $M$.
Top: Parareal-Euler ($\mathcal{F}$=backward Euler) and Parareal-TR/BDF2 ($\mathcal{F}$=TR/BDF2).
Bottom left side: Parareal-TR ($\mathcal{F}$=trapezoidal rule) and Parareal-Gauss4 ($\mathcal{F}$=forth-order Gauss).
Bottom right side: Parareal-fEuler ($\mathcal{F}$=forward Euler) and Parareal-4ERK ($\mathcal{F}$=forth-order explicit Runge-Kutta).}
\label{Fig:com}
\end{figure}
\end{rem}
\section{Numerical Experiment}\label{sec:num}
In this section we verify the convergence and efficiency of the Parareal-CG algorithm. 
In all the computations, the initial iteration for the parareal algorithm is chosen randomly
and the iteration process stops when the following tolerance is obtained:
\begin{equation}
\max_n\|u^{k+1}_n-u^k_n\|_{\infty}\leq 10^{-10}.
\end{equation}
Moreover, the absolute error and iteration error in all experiments are defined by
\begin{equation}\begin{split}
\text{Absolute error}:\max_n\|u^{k+1}_n-u(T_n)\|_{\infty}.\quad
\text{Iteration error}:\max_n\|u^{k+1}_n-u^k_n\|_{\infty}.
\end{split}\end{equation}
%
\begin{ex}[Near Earth satellite motion integration problem]
We take a two-body motion in low Earth orbit (LEO) with just mutual gravitational attraction for the first example.
The system of the second order three-dimensional equations is
\begin{equation}\label{2body}
\left\{\begin{split}
&x''(t)=\;-\frac{\mu}{r(t,x,y,z)^3}x(t),\quad t\in[0,50],\\
&y''(t)=\;-\frac{\mu}{r(t,x,y,z)^3}y(t),\quad t\in[0,50],\\
&z''(t)=\;-\frac{\mu}{r(t,x,y,z)^3}z(t),\quad t\in[0,50],
\end{split}\right.
\end{equation}
where $x$, $y$ and $z$ are the three coordinates in some Earth-centered inertial reference frame;
$r$ is the distance between the two bodies defined by
\begin{equation*}\begin{split}
r(t,x,y,z) =&\; \sqrt{x(t)^2+y(t)^2+z(t)^2};
\end{split}\end{equation*}
$\mu=3.986\times 10^5 km^3/s^2$ is the Earth gravitational constant.
We formulate the equations as a first-order system of six-dimensional differential equations,
We formulate the equations as a first-order system of six-dimensional differential equations,
where the solution is uniquely determined by the initial position and velocity
\begin{equation*}\begin{split}
[x(0), y(0),z(0)]&=\;[464.856, 6667.880, 574.231] km,\\
[x'(0), y'(0),z'(0)] &=\; [-2.8381188, -0.7871898, 7.0830275] km/s.
\end{split}\end{equation*}
We obtain the exact solution of the equation \eqref{2body} by
solving the two-body Keplerian motion using the $F$ and $G$ approach in \cite{SJ2003}.
\end{ex}

We compare the accuracy of Parareal-CG, Parareal-Euler, Parareal-TR/BDF2 and Parareal-Gauss4 algorithms
revisited in Remark \ref{rem:com} for solving the example.
In each coarse grid, we fix $\Delta T = 0.25$
with $M=6$ for Parareal-CG algorithm and $J=6$ for other parareal algorithms.
Then we perform the absolute errors and iteration errors of the parareal algorithms in Figure \ref{Fig:ex1}.
In this way, we can draw the following conclusions by the convergence behaviors in the figures.
\begin{itemize}
\item In comparison to the Parareal-Euler and Parareal-TR, the Parareal-CG algorithm utilizes less iterations.
While the Parareal-Gauss4 algorithm uses the same number of iterations as Parareal-CG algorithm.
\item Using the same number of points $M=J=6$ in every coarse grid,
   the Parareal-CG algorithm has a higher accuracy than the three others.
\end{itemize}
\begin{figure}
\centerline{\includegraphics[width=0.45\textwidth]{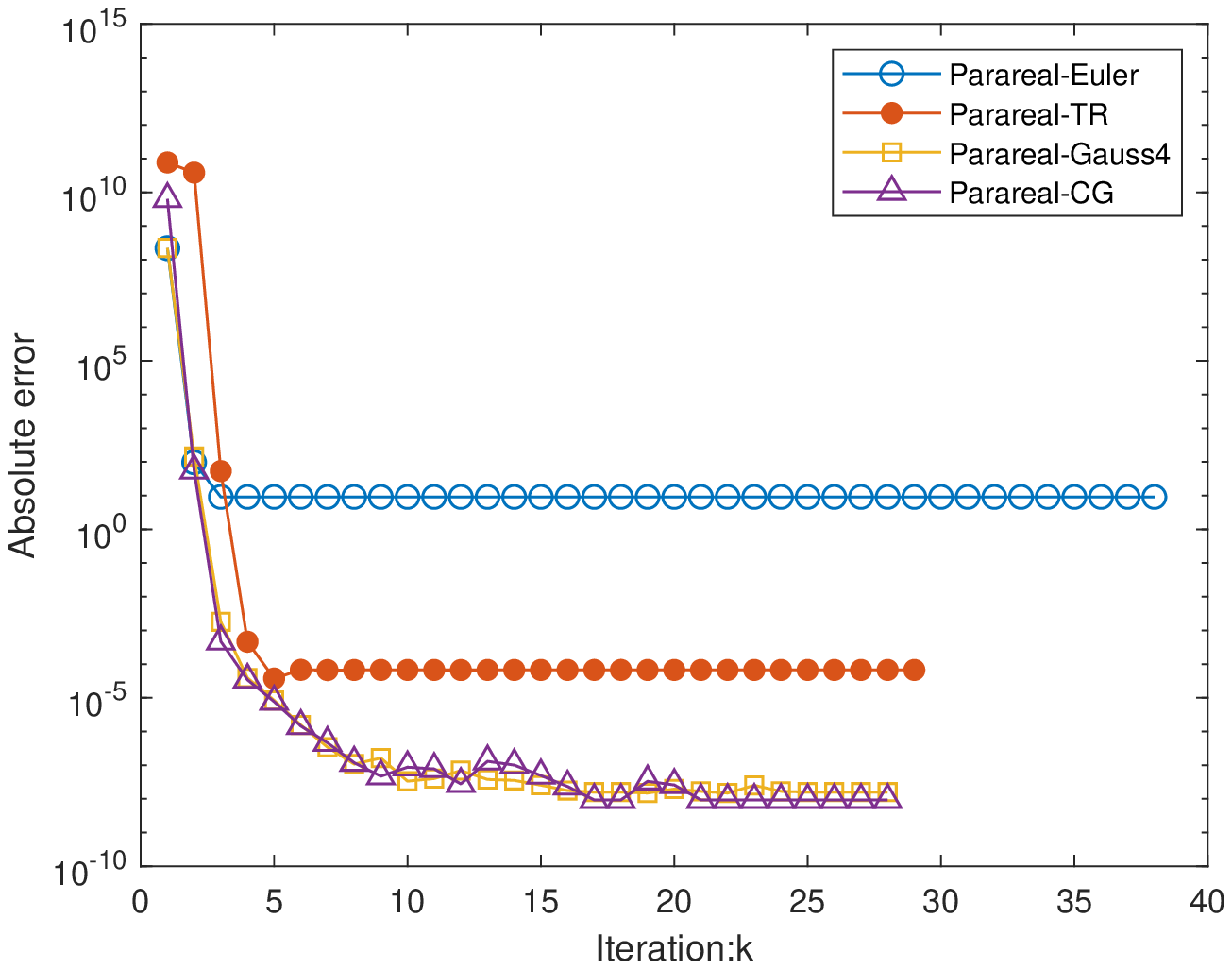}
\includegraphics[width=0.45\textwidth]{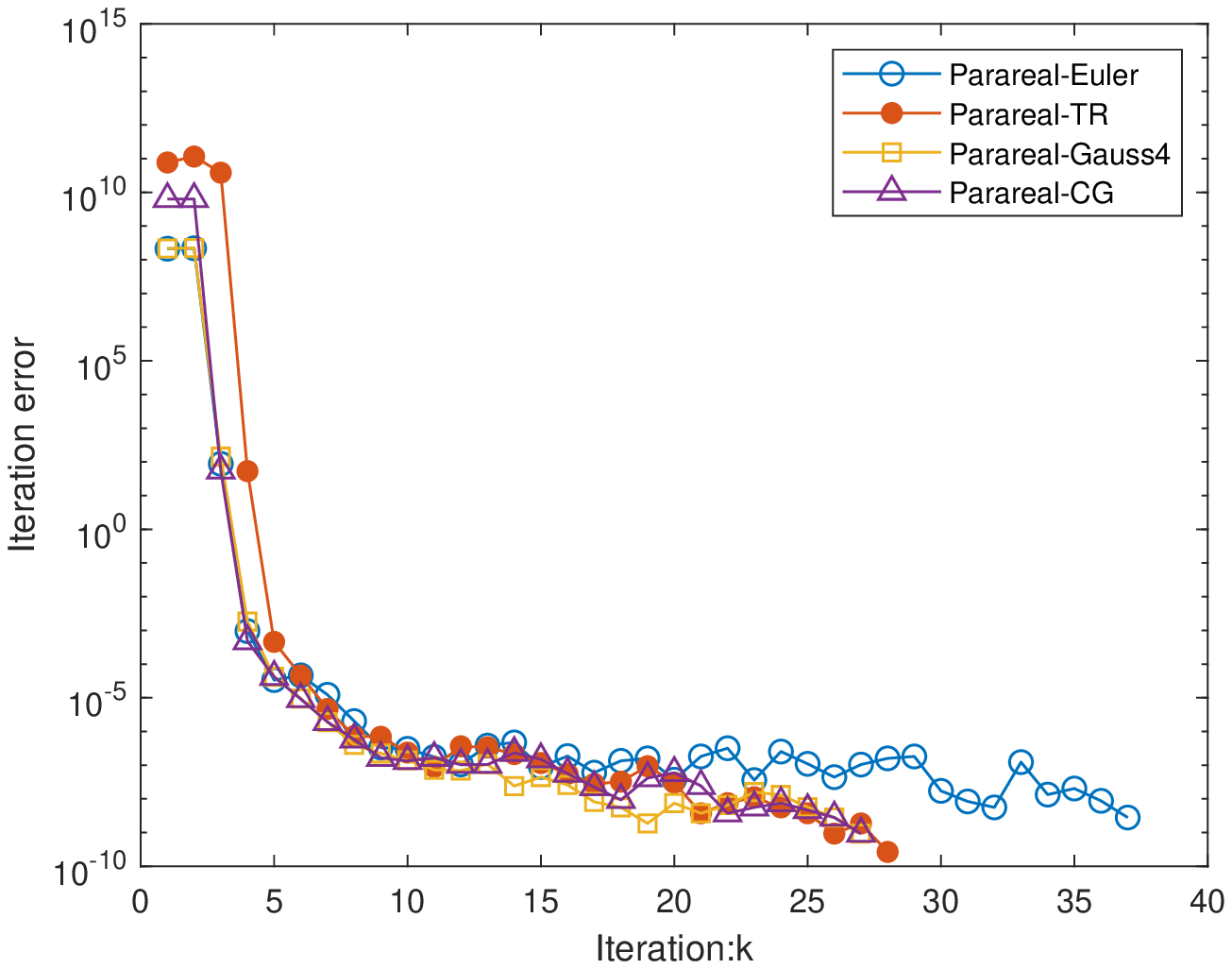}}
\caption{Comparisons of the absolute errors (left) and iteration errors (right)
for the Parareal-Euler, Parareal-TR and Parareal-Gauss4 and Parareal-CG algorithms.}
\label{Fig:ex1}
\end{figure}


\begin{ex}[Burgers' equation]\label{ex:Bur}
Consider the 1D Burgers' equation with initial and boundary condition as following
\begin{equation}\label{ex2:eq}
\left\{\begin{aligned}
&\;\frac{\partial u}{\partial t}-\nu \frac{\partial^2 u}{\partial x^2}+u\frac{\partial u}{\partial x}=0,
\quad\;\; (x,t)\in [0,2)\times[0,4],\\
&\;u(x,0)=\frac{2\nu\pi\sin(\pi x)}{\alpha+\cos(\pi x)}, \quad\quad\; x\in[0,2),\\
&\;u(0,t)=u(2,t)=0,\quad \quad \;\;\;\;\; t\in[0,4],
\end{aligned}\right.
\end{equation}
in our computations, we choose $\alpha=2$ and test $\nu=0.05,0.005$.
The exact solution of the equation is
\begin{equation}
u(x,t)=\frac{2\nu\pi\exp(-\pi^2\nu t)\sin(\pi x)}{\alpha+\exp(-\pi^2\nu t)\cos(\pi x)}.
\end{equation}
\end{ex}
We divide the spatial domain $x\in[0,2)$ into an $N_x$ mesh uniformly with $\Delta x = \frac{2}{N_x}$,
in this way, we have $x_j=j\Delta x$, $j=0,1,\cdots,N_x-1$.
Applying the fourth-order compact finite difference scheme
to approximate $\frac{\partial u}{\partial x}$ and $\frac{\partial^2 u}{\partial x^2}$, we have
\begin{small}\begin{equation}\begin{split}
&\frac{1}{6}\frac{\partial u}{\partial x}(x_{j+2},t)+\frac{2}{3}\frac{\partial u}{\partial x}(x_{j+1},t)
+\frac{1}{6}\frac{\partial u}{\partial x}(x_{j},t)=\;\frac{1}{2\Delta x}\left(u(x_{j+2},t)-u(x_{j},t)\right),\\
&\frac{1}{12}\frac{\partial^2 u}{\partial x^2}(x_{j+2},t)+\frac{5}{6}\frac{\partial^2 u}{\partial x^2}(x_{j+1},t)
+\frac{1}{12}\frac{\partial^2 u}{\partial x^2}(x_{j},t)=\;\frac{1}{\Delta x^2}\left(u(x_{j+2},t)-2u(x_{j+1},t)+u(x_{j},t)\right).\\
\end{split}\end{equation}\end{small}
Define the vectors
\begin{equation}\begin{split}
\mathbf{\dot{u}}=\;&\left[\frac{\partial u(x_{0},t)}{\partial t},
\frac{\partial u(x_{1},t)}{\partial t},\cdots, \frac{\partial u(x_{N_x-1},t)}{\partial t}\right]^{\top},\\
\mathbf{u}=\;&\left[u(x_{0},t), u(x_{1},t),\cdots, u(x_{N_x-1},t)\right]^{\top}.
\end{split}\end{equation}
Then combining the compact finite difference scheme with the period boundary condition, the Burgers' equation \eqref{ex2:eq}
has the following spatial semidiscrete scheme
\begin{equation}
\mathbf{\dot{u}}+A_1\mathbf{u}+\mathbf{u}\cdot(A_2\mathbf{u})=0,
\end{equation}
where the coefficient matrices $A_1$ and $A_2$ are
\begin{equation*}\begin{aligned}
&A_1 =\;\mathop{\renewcommand\arraystretch{1.5}
-\frac{\nu}{\triangle x^2}\;
\begin{bmatrix}
\frac{5}{6} &\frac{1}{12} & & & \frac{1}{12}\\
\frac{1}{12} & \frac{5}{6} & \frac{1}{12} & & \\
& \ddots & \ddots & \ddots &\\
& & \frac{1}{12} & \frac{5}{6} & \frac{1}{12}\\
\frac{1}{12} & & & \frac{1}{12} &\frac{5}{6}
\end{bmatrix}^{-1}
\begin{bmatrix}
-2 &1 & & & 1\\
1 & -2 & 1 & & \\
& \ddots & \ddots & \ddots & \\
& & 1 & -2 & 1\\
1 & & & 1 &-2
\end{bmatrix},}\\
&\quad A_2 =\; \mathop{\renewcommand\arraystretch{1.5}
\frac{1}{2\triangle x}
\begin{bmatrix}
\frac{2}{3} &\frac{1}{6} & & & \frac{1}{6}\\
\frac{1}{6} & \frac{2}{3} & \frac{1}{6} & & \\
& \ddots & \ddots & \ddots &\\
& & \frac{1}{6} & \frac{2}{3} & \frac{1}{6}\\
\frac{1}{6} & & & \frac{1}{6} &\frac{2}{3}
\end{bmatrix}^{-1}
\begin{bmatrix}
0 &1 & & & -1\\
-1 & 0 & 1 & & \\
& \ddots & \ddots & \ddots & \\
& & -1 & 0 & 1\\
1 & & & -1 &0
\end{bmatrix}.}
\end{aligned}\end{equation*}
The three figures in Fig \ref{Fig:ex2} show
the dependence of the convergence rate of the Parareal-CG algorithm
on $\Delta T$, $\Delta x$, $M$, respectively.
\begin{figure}[!ht]
\centerline{\includegraphics[width=0.4\textwidth]{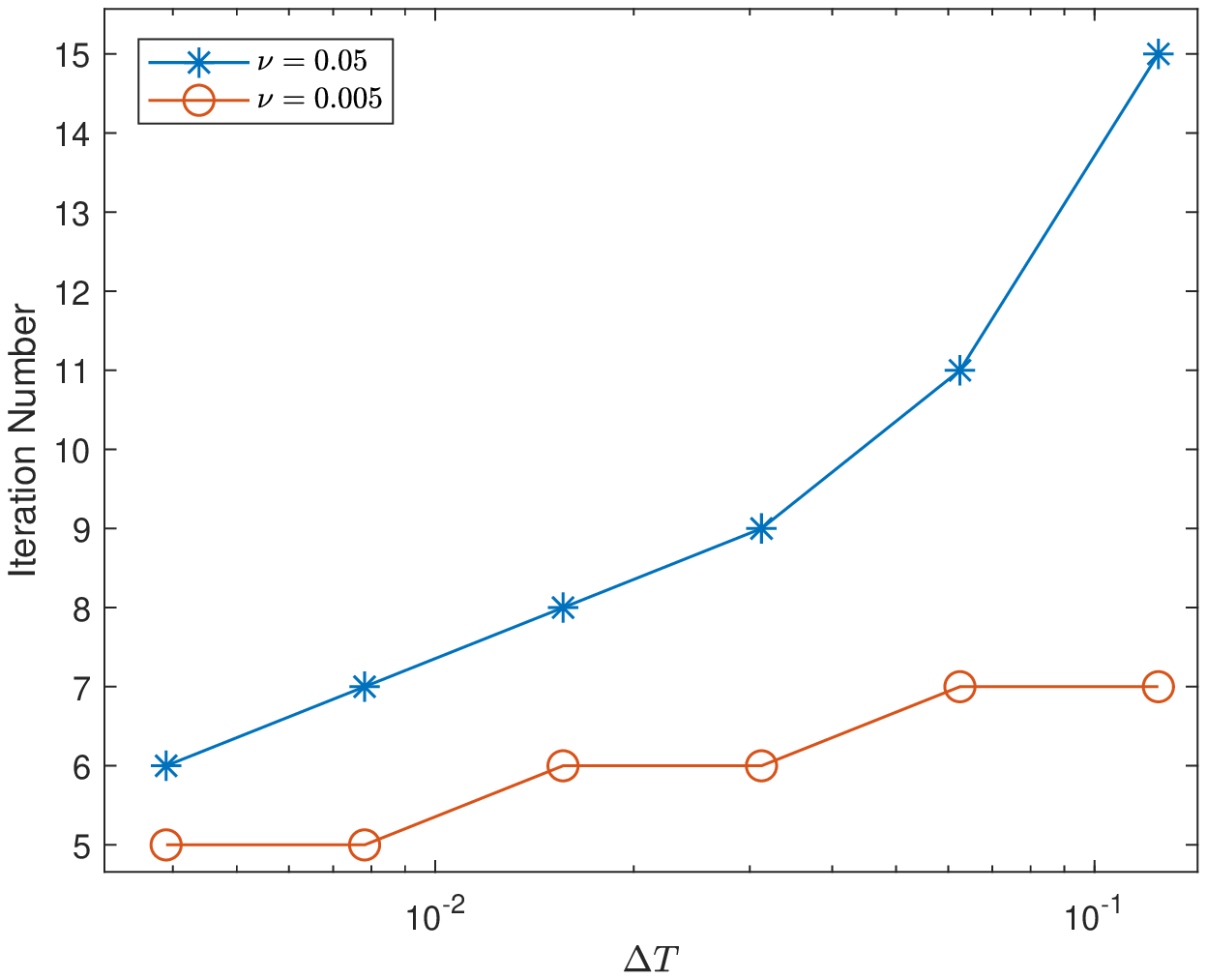}}
\centerline{\includegraphics[width=0.4\textwidth]{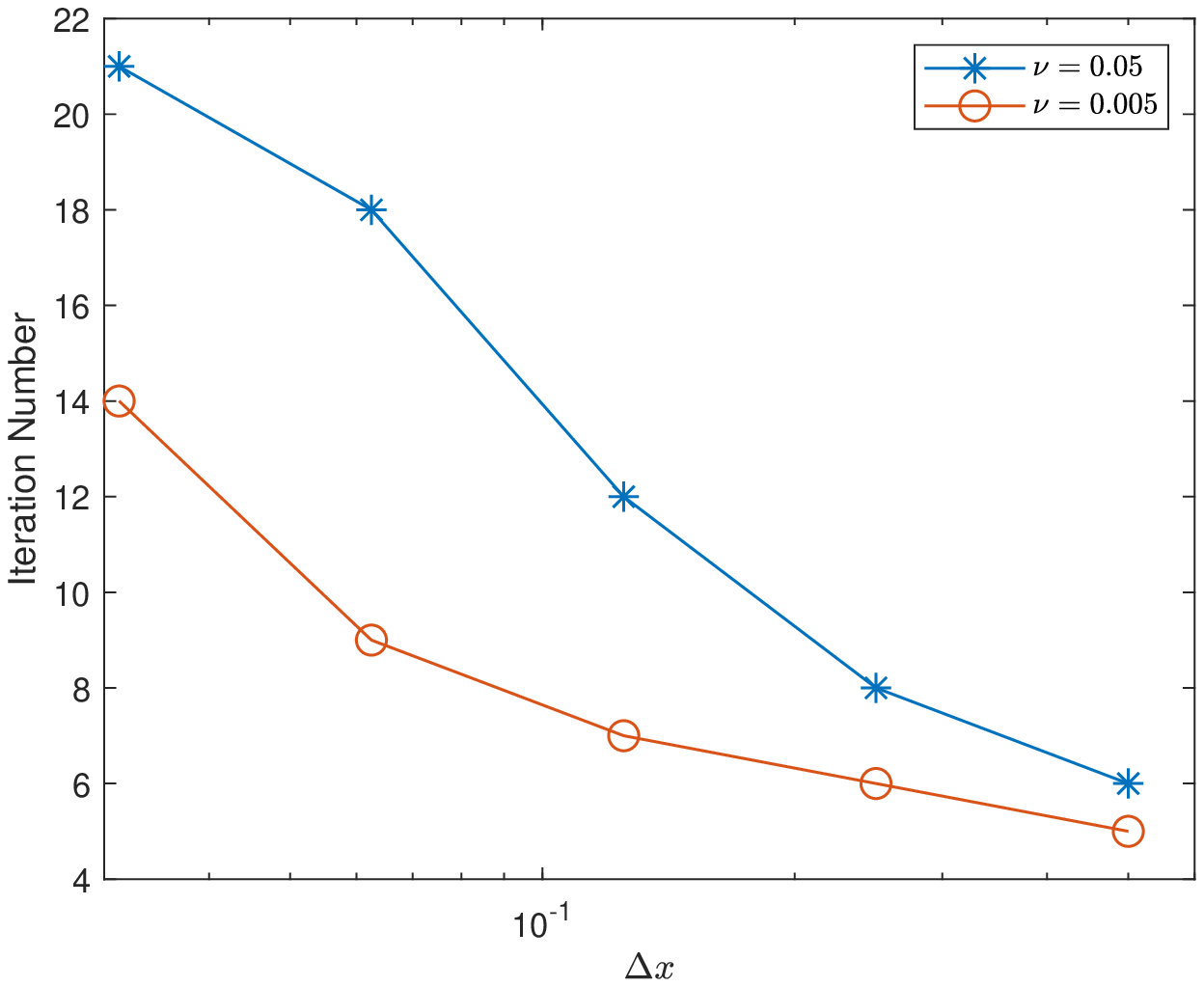}
\includegraphics[width=0.4\textwidth]{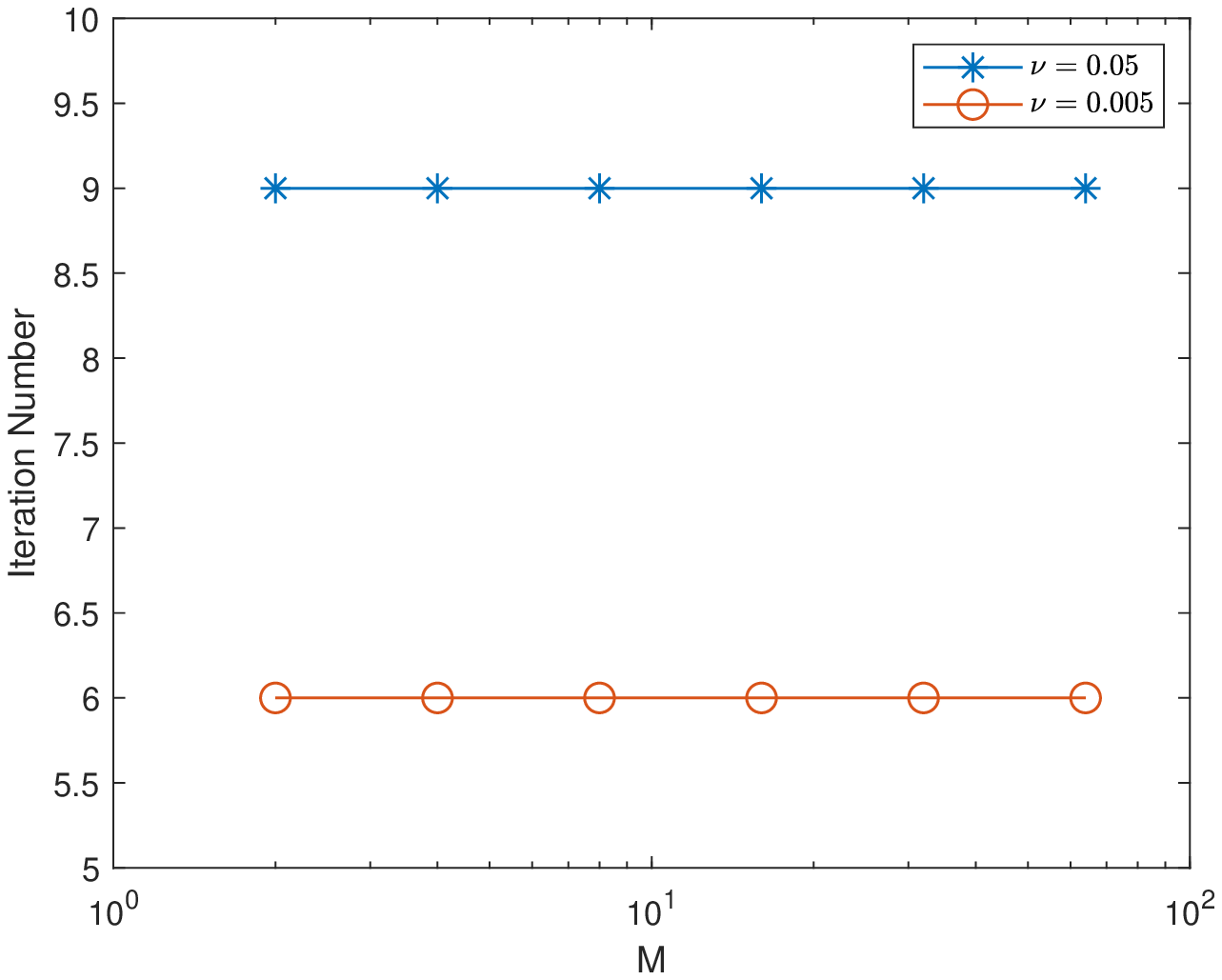}}
\caption{Top: The semi-log plot of the dependency of the Parareal-CG algorithm's convergence rate
on $\Delta T$, $\Delta x= 1/4$, $M=4$, $\Delta T$ ranges from $2^{-3}$ to $2^{-8}$.
Bottom left side: The semi-log plot of the dependency of the Parareal-CG algorithm's convergence rate
on $\Delta x$, $\Delta T= 1/64$, $M=4$, $\Delta x$ ranges from $2^{-1}$ to $2^{-5}$.
Bottom right side: The semi-log plot of the dependency of the Parareal-CG algorithm's convergence rate
on $M$, $\Delta T= 1/32$, $\Delta x=1/4$, $M$ ranges from $2^{1}$ to $2^{6}$.}
\label{Fig:ex2}
\end{figure}
We can draw the following conclusions by the figure.
\begin{itemize}
\item 
Each of the $\nu$ possesses robust convergence rate with respect to the change of $\Delta T$.
The convergence factor decreases as $\Delta T$ increases.
\item Each of the $\nu$ possesses robust convergence rate with respect to the change of $\Delta x$.
The convergence factor increases as $\Delta x$ varies from small to large
since the eigenvalues of $A_1$ and $A_2$ increases as $\Delta x$ reduces.
\item The convergence rate is insensitive to the choice of $M$ with $\Delta T=1/32$, which implies the high accuracy of the
Chebyshev-Gauss spectral collocation method.
\item All the experiments $\nu=0.005$ needs less iteration number than $\nu=0.05$.
\end{itemize}

\section{Conclusion}\label{sec:conclusion}
In this paper, we propose the Parareal-CG method for solving time-dependent differential equations,
where the coarse propagator $\mathcal{G}$ is fixed by backward Euler method
and the fine propagator $\mathcal{F}$ is chosen to be Chebyshev-Gauss spectral collocation method.
The algorithm does have a convergence factor around $0.333$,  
although the number of Chebyshev-Gauss points $M$ needs to be somewhat large.
The spectral radius of the matrix $A$ and the $\Delta T$
provide the lower bound of $M$ that guarantees such a expective convergence factor.
Some numerical experiments illustrate the accuracy and convergency of the presented algorithm.
\section*{Acknowledgments}
We would like to thank the anonymous reviewers for their valuable suggestions,
which helped us to improve this article greatly.
This work was partially supported by the Postgraduate Scientific Research Innovation Project of Hunan Province (No.CX20210012).

\end{document}